\documentclass[a4paper,12pt]{article}

\usepackage{amssymb, amsmath, amsthm, mathtools}
\usepackage{authblk}

\newtheorem{thm}{Theorem}[section]
\newtheorem{lem}[thm]{Lemma}

\newtheorem{prop}[thm]{Proposition}
\newtheorem{conj}[thm]{Conjecture}

\theoremstyle{definition}
\newtheorem{ex}[thm]{Example}
\newtheorem{rem}[thm]{Remark}
\newtheorem{defn}[thm]{Definition}

\DeclareMathOperator{\h}{\mathbf{h}} 

\newcommand{\abs}[1]{\lvert#1\rvert}

\newcommand{\norm}[1]{\lVert#1\rVert}

\newcommand{\dpair}[2]{\langle#1, #2\rangle}

\newcommand{\R}{\mathbb{R}}

\newcommand{\N}{\mathbb{N}}  

\numberwithin{equation}{section}

\title{\bf Metric functionals and weak convergence}

\author[1]{Armando W. Guti\'{e}rrez\thanks{
	Supported by Horizon Europe (MOPO project, grant 101095998).}}
\affil[1]{\itshape\small VTT Technical 
Research Centre of Finland \protect\\ 
\emph{armando.w.gutierrez@vtt.fi}}

\author[2]{Olavi Nevanlinna}
\affil[2]{\itshape\small Department of 
Mathematics and Systems Analysis, 
Aalto University \protect\\ 
\emph{olavi.nevanlinna@aalto.fi}}
    
\date{}

\begin{document}

\maketitle

\begin{abstract}
We introduce a notion of weak convergence in arbitrary metric spaces. Metric 
functionals are key in our analysis: weak convergence of sequences in a given 
metric space is tested against all the metric functionals defined on said space. 
When restricted to bounded sequences in normed linear spaces, we prove that 
our notion of weak convergence agrees with the standard one.
\end{abstract}

\section{Introduction}

A natural question to ask is whether there exists a notion of weak convergence 
which may be valid in \emph{all} metric spaces. If such a concept exists, one may
naturally expect that it agrees with the standard weak convergence in all normed 
linear spaces. 

While investigating certain fixed point problems, T.-C.~Lim \cite{Lim1977}
introduced an interesting notion of convergence in metric spaces, the so-called 
$\Delta-$convergence. This concept, however, does not agree with the standard weak 
convergence in \emph{some} normed linear spaces. This is discussed in the end of
Section~\ref{sec:WC}. 

Lim's work has nevertheless motivated some investigations on weak topologies in 
{CAT}($0$) spaces. In these spaces, one can define metric projections on compact 
geodesics. This has led researchers to consider a notion of weak convergence 
that agrees with the standard one in Hilbert spaces. See the works 
\cite{Kirk2008}, \cite{Espinola-Fernandez2009},
\cite{Bacak2023}, and \cite{Lytchak-Petrunin2023}. 

In this note, we present a novel notion of weak convergence in metric spaces. Our 
approach uses metric functionals. 

\begin{defn}\label{def:WeakConv}
Let $(X,d)$ be a metric space. We say that a sequence of points $(x_n)_{n\geq 1}$ 
in $X$ converges \emph{$d-$weakly} to a point $z$ in $X$ if we have
\begin{equation}\label{eq:LiminfIneq}
	\liminf_{n\to\infty}\,\h(x_n) \geq \h(z),
\end{equation}
for all metric functionals $\h$ on $X$. We use the notation 
$x_n \xrightharpoonup{\;d\;} z$.
\end{defn}

Our concept of $d-$weak convergence can be applied to nets. We discuss here
$d-$weak convergence of sequences only to simplify the exposition of our ideas.
 
We also want to emphasize that $d-$weak convergence fully depends on the given 
metric $d$. The metric $d$ determines all the corresponding metric 
functionals $\h$ that are essential to test the inequality (\ref{eq:LiminfIneq}). 

The concept of a metric functional is fairly known in some mathematical 
disciplines. Nonetheless, we recall such a concept and provide new properties in 
Section~\ref{sec:MF}.

We proceed now to provide evidence that supports our choice of 
Definition~\ref{def:WeakConv} as a suitable candidate for weak convergence in 
metric spaces. The following nontrivial statements concern $d-$weak convergence
of sequences in normed linear spaces.

\begin{thm}\label{thm:unique}
Let $(X,\norm{\cdot})$ be a normed linear space. Let $d$ be the metric on $X$ 
induced by the norm $\norm{\cdot}$. If a sequence of points $(x_n)_{n\geq 1}$ 
in $X$ converges $d-$weakly, then its limit is unique.
\end{thm}

The following result shows the connection of $d-$weak convergence with the standard 
weak convergence in normed linear spaces. Namely, they agree for bounded 
sequences in \emph{all} normed linear spaces.

\begin{thm}\label{thm:dWeakANDstWeak}
Let $(X,\norm{\cdot})$ be a normed linear space. Let $d$ be the metric on $X$ 
induced by the norm $\norm{\cdot}$. Suppose that $(x_n)_{n\geq 1}$ is a bounded 
sequence in $X$ and $z$ is a point in $X$. Then, $(x_n)_{n\geq 1}$ converges 
$d-$weakly to $z$ if and only if $(x_n)_{n\geq 1}$ converges weakly to $z$.
\end{thm}

One may wonder whether unbounded sequences in normed linear spaces converge 
$d-$weakly. We can prove that this does not happen in some normed linear spaces.

\begin{thm}\label{thm:UnbNoWeak}
Suppose that $X$ is $\ell_1$, or $C[0,1]$, or a normed linear space whose dual 
is strictly convex. If $d$ is the corresponding induced metric, then unbounded 
sequences in $X$ do not converge $d-$weakly.
\end{thm}

\begin{conj}
There is no normed linear space where unbounded sequences converge $d-$weakly.
\end{conj}

Without a proof of the above conjecture, it is unclear how $d-$weak convergence 
behaves on linear combinations of two $d-$weakly convergent sequences. 
Nevertheless, we can offer the following.

\begin{thm}\label{thm:dWeakOfLinComb}
Let $(X,\norm{\cdot})$ be a normed linear space. Let $d$ be the metric on $X$ 
induced by the norm $\norm{\cdot}$. Suppose that $(x_n)_{n\geq 1}$ and 
$(y_n)_{n\geq 1}$ are two sequences in $X$. If there are two vectors $u$ and $v$ 
in $X$ such that $x_n \xrightharpoonup{\;d\;} u$ and  $d(y_n,v)\rightarrow 0$, 
then we have
$$
	sx_n + ty_n \xrightharpoonup{\;d\;} su+ tv,
$$
for all real numbers $s$ and $t$. 
\end{thm}

All the results presented above examine $d-$weak convergence of sequences in the 
whole normed linear space. Proper subsets become legit metric spaces when we 
equip them with the subspace topology. With this in mind, we can produce simple 
examples of metric spaces where unbounded sequences converge $d-$weakly, or 
where $d-$weakly convergent sequences have more than one limit. An example where
both situations happen is the following. Equip the set 
$X=\{0,e_1, 2e_2, 3e_3,\dots\}$ with the metric $d$ that is induced by the norm 
of $\ell_1$. We can verify that the unbounded sequence $(ne_n)_{n\geq 1}$ in $X$ 
converges $d-$weakly to every point in $X$. More examples are presented in 
Section~\ref{sec:WC}.

The following result states that the set of all $d-$weak limits of a sequence in
a $W-$convex metric space is also $W-$convex. The concept of $W-$convexity was 
introduced by W.~Takahashi in \cite{Takahashi1970}. We recall this concept in 
Definition~\ref{def:TakahashiConv}. Briefly we mention that every metric space 
that admits a conical geodesic bicombing is $W-$convex.

\begin{thm}\label{thm:dWeakInWconvex}
Let $(X,d)$ be a $W-$convex metric space. Then, the set of all $d-$weak limits of
a sequence in $X$ is $W-$convex and closed.
\end{thm}

\section{Metric functionals}\label{sec:MF}

\subsection{Where do they appear?}

Metric functionals have become valuable tools in investigations of iterative 
processes in metric structures. 

Fixed point problems in metric spaces are studied by means of metric functionals. 
Briefly, instead of looking directly for a fixed point in the space, one rather 
seeks a metric functional that remains (sub)invariant by the iterations of the 
mapping or family of mappings involved in the problem. These ideas have been put 
into action in 
\cite{Gaubert-Vigeral2012}, \cite{Lemmens-Lins-Nussbaum2018}, 
\cite{Gutierrez-Karlsson2021}, \cite{Lins2022}, \cite{Karlsson2023}, 
\cite{Gutierrez-Nevanlinna2024}. 

Metric functionals also have been applied to study noncommuting random products. 
There is always a metric functional that determines the direction along which a 
random product grows \cite{Gouezel-Karlsson2020}. Motivated by this result, 
A.~Avelin and A.~Karlsson \cite{Avelin-Karlsson2022} have used metric 
functionals to study a cut-off phenomenon associated with the dynamics of deep 
neural networks.

Metric functionals have appeared recently in studies on Lipschitz-free spaces 
\cite{Megrelishvili2025} and Hermitian symmetric spaces 
\cite{Chu-Cueto-Lemmens2024}.

\subsection{What are they like?}\label{ssec:Construction}

Let $(X,d)$ be a metric space. We choose a basepoint $o$ in $X$ and for each point 
$w$ in $X$ we consider the funtional 
\begin{equation}\label{eq:int_mf}
	\h_w(x) = d(x,w)-d(o,w),\;\text{for all}\;x\;\text{in}\;X.
\end{equation}
We equipped the product space $\R^X$ with the topology of pointwise convergence 
and denote by $X^{\vee}$ the subset of $\R^X$ containing all the functionals of 
the form (\ref{eq:int_mf}). The closure of $X^{\vee}$ is denoted by 
$X^{\diamondsuit}$ and each element $\h$ in it is called a 
\emph{metric functional}.

The next statement is obtained by applying properties of the metric $d$ to the 
functionals (\ref{eq:int_mf}).
\begin{prop}[{\cite[Chapter~3]{GutierrezThesis}}]
The following properties hold:
\begin{enumerate}
	\item The mapping $w\mapsto\h_w$ from $X$ to $X^\vee$ is injective and 
	continuous.
	\item The space $X^{\diamondsuit}$ is compact and Hausdorff. In particular, 
	for every metric functional $\,\h \in X^\diamondsuit$ there exists a net of 
	points $(w_\alpha)$ in $X$ such that 
	$$
		\h(x) = \lim_{\alpha}\,\h_{w_\alpha}(x),\;\text{for all}\;x\;\text{in}\;X. 
	$$
	\item Every metric functional $\,\h \in X^\diamondsuit$ is $1-$Lipschitz; 
	that is, we have
	$$
		\abs{\h(x)-\h(y)} \leq d(x,y),\;\text{for all}\;x\;
		\text{and}\;y\;\text{in}\;X. 
	$$
\end{enumerate}
\end{prop}

\begin{rem}
The set $X^\diamondsuit$ is sometimes called the \emph{metric compactification} 
of $X$. The spaces $X$ and $X^\vee$, however, need not be homeomorphic. In fact, 
for the unbounded sequence $(ne_n)_{n\geq 1}$ in $X=\ell_1$ we have
$\h_{ne_n} \rightarrow \h_0$. Another example is the CAT(0) space given in 
\cite[Example~3.2]{Karlsson2022}. If one insists on having $X^\diamondsuit$ as a 
standard topological compactification of $X$, some conditions on $X$ must hold. 
Such conditions are discussed in \cite[Theorem~2.1]{Daniilidis-et-al2024}. 
\end{rem}

The existence of the continuous injection 
$ 
X\xhookrightarrow{} X^\vee \subset X^\diamondsuit
$
motivates us to call all the elements in $X^\vee$ \emph{internals}. In this way, 
an analog of the Banach-Alaoglu theorem reads: \emph{Given a net of internals
$(\h_{w_\alpha})$, there exists a subnet $(\h_{w_\beta})$ and a metric functional
$\h \in X^\diamondsuit$ such that $\h_{w_\beta}(x) \rightarrow \h(x)$ for all 
$x$ in $X$}. This type of convergence may be seen as an analog of the weak-star 
convergence of continuous linear functionals. Notice that all metric functionals 
are continuous. In the linear theory, however, continuity of 
linear functionals must be indicated explicitly. 

There exist explicit formulas for all metric functionals on infinite dimensional
$\ell_p$ spaces with $1\leq p<\infty$, see \cite{Gutierrez2019-2}. Explicit 
formulas for an important class of metric functionals on $C(K)$ are shown in 
\cite{Walsh2018}.

\begin{rem}
Limits of the type $\h_{w_n} \rightarrow \h$ behave differently in $\ell_1$ and 
$\ell_p$ with $1<p<\infty$. The functional identically zero, denoted by $\mathbf{0}$, is a metric 
functional on each $\ell_p$ with $1<p<\infty$, but is not on $\ell_1$. Define $w_n:=ne_n$ for all 
$n\geq 1$. In $\ell_p$ we have $\h_{w_n}\rightarrow \mathbf{0}$, whereas in 
$\ell_1$ we have $\h_{w_n}\rightarrow \h_0$.
\end{rem}

\subsection{What properties do they have?}

The space $X^\diamondsuit$ is not always metrizable. It becomes metrizable when
$X$ is separable. To be precise, we have the following.
 
\begin{prop}
Let $(X,d)$ be a separable metric space. Then, every sequence of points 
$(x_n)_{n\geq 1}$ in $X$ has a subsequence $(w_{i})_{i\geq 1}$ such that the 
sequence of internals $(\h_{w_i})_{i\geq 1}$ converges in $X^\diamondsuit$.   
\end{prop}
	
\begin{proof}
We fix a point $o$ in $X$ and consider a countable dense subset $\{y_1, y_2, ...\}$ 
of $X$. Let $(x_n)_{n\geq 1}$ be a sequence in $X$. Each internal $\h_{x_n}$ 
satisfies the property $\abs{\h_{x_n}(y_k)} \leq d(o, y_k)$ for all positive 
integers $k$. By using Cantor's diagonal argument, we obtain a subsequence 
$(w_i)_{i\geq 1}$ of $(x_n)_{n\geq 1}$ such that for every positive integer $k$ 
the sequence $(\h_{w_i}(y_k))_{i\geq 1}$ converges to a real number, say $r_k$, 
in the closed interval $[-d(o,y_k),\,d(o,y_k)]$. Now, since $X^\diamondsuit$ is a 
compact Hausdorff space, the sequence of internals $(\h_{w_i})_{i\geq 1}$ has a 
limit point $\h\in X^\diamondsuit$ so that $\h(y_k)=r_k$ for all $k$.

Let $x$ be a point in $X$ and $\epsilon$ be a positive real number. Then, there 
is an integer $m\geq 1$ such that $d(x,y_m) < \epsilon /4$. The identity 
$$
	\lim_{i\to\infty}\h_{w_i}(y_m)=\h(y_m)
$$ 
implies that there is an integer $i_0 \geq 1$ such that 
$\abs{\h_{w_i}(y_m) - \h(y_m)} < \epsilon / 2$ for all $i\geq i_0$. Thus, for 
every $i\geq i_0$ we have
\begin{equation*}
	\abs{\h_{w_i}(x) - \h(x)}
	\leq 2\,d(x,y_m) + 
	\abs{\h_{w_i}(y_m) - \h(y_m)} < \epsilon\,. 
\end{equation*}
\end{proof}	

Sometimes, the space $X^\diamondsuit$ consists of only internals. For example, 
if $X$ is a compact metric space, then we have $X^\diamondsuit = X^\vee$. The 
same phenomenon can occur for noncompact sets; for example when $X$ is the 
closed unit ball of $\ell_1$. 

There are metric spaces $X$ for which we have 
$X^\diamondsuit = X^\vee\cup\{\mathbf{0}\}$, where $\mathbf{0}$ is the functional
identically zero. For example, let $X$ be the set of real numbers and equip it 
with the metric $d=\abs{\cdot-\cdot}^p$ with $0<p<1$.

The following property holds in all metric spaces, and is a straightforward 
consequence of the definition of metric functionals.

\begin{prop}
Let $(X,d)$ be a metric space. Then, for every point $w$ in $X$ we have 
$$
	d(o,w)=\max_{\h\in X^\diamondsuit}\h(w).
$$
\end{prop}

\begin{proof}
The claim is trivial when $w$ is the basepoint $o$ because at this point every 
metric functional vanishes. So, let $w$ be a point in $X$ with $w\neq o$. Since 
every metric functional is $1-$Lipschitz, we have $\h(w) \leq d(o,w)$ for all 
$\h$ in $X^\diamondsuit$, and the equality holds for the internal $\h_o$.
\end{proof}

\begin{rem}
The Hahn-Banach theorem in normed linear spaces implies that for every vector $v$ 
one has $\norm{v}=\max_{f}\abs{f(v)}$, where the maximum is taken over all 
continuous linear functionals of norm at most $1$. A weaker version of this will 
be shown in Proposition~\ref{prop:NormAsSupOfBusm}, where the Hahn-Banach theorem 
is not used.
\end{rem}

Next we discuss some properties that metric functionals possess when the given 
metric space has a certain convex structure.

\begin{defn}\label{def:TakahashiConv}
A metric space $(X,d)$ is said to be $W-$convex if there exists a mapping $W$ 
from $X\times X\times[0,1]$ to $X$ such that 
$$
	d(z,W(x,y,t)) \leq (1-t)d(z,x) + td(z,y),
$$ 
for all points $x$, $y$, and $z$ in $X$ and all $t$ in $[0,1]$. This concept was 
introduced by W.~Takahashi in \cite{Takahashi1970} to study some fixed point 
theorems.
\end{defn}

Every convex subset of a normed linear space is $W-$convex as a 
metric space. To see this, one considers the mapping $W(x,y,t)=(1-t)x+ty$. 

Nontrivial examples of $W-$convex metric spaces appear naturally in
geometric group theory. Metric spaces that admit a \emph{conical bicombing} are 
such examples. A standard reference on conical bicombings is the paper
\cite{Descombes-Lang2015}. It seems that the authors of \cite{Descombes-Lang2015} 
were unaware of Takahashi's notion of convexity.

A real-valued functional $f$ defined on a $W-$convex metric space $(X,d)$ is said 
to be $W-$convex if for every $x$ and $y$ in $X$ and $t$ in $[0,1]$ we have
$$
	f(W(x,y,t)) \leq (1-t)f(x)+tf(y).
$$

\begin{prop}\label{prop:ConvexMF}
Let $(X,d)$ be a $W-$convex metric space. Then, every metric functional on $X$ is 
$W-$convex.
\end{prop}

\begin{proof}
Let $\h$ be a metric functional on $X$. Suppose that $(w_\alpha)$ is a net of 
points in $X$ so that $\h_{w_\alpha} \rightarrow \h$. Let $x$ and $y$ be two 
points in $X$ and $t$ in [0,1]. Since $(X,d)$ is $W-$convex, for all points 
$w_\alpha$ we have
$$
	d(W(x,y,t),w_\alpha) - d(o,w_\alpha) \leq (1-t)\h_{w_\alpha}(x) 
	+ t\h_{w_\alpha}(y).
$$
Thus, we get $\h(W(x,y,t)) \leq (1-t)\h(x) + t\h(y)$.
\end{proof}

Now we discuss properties of an important class of metric functionals in normed
linear spaces. The obvious basepoint that we choose is the origin.
 
\begin{defn}\label{def:Busm}
Let $(X,\norm{\cdot})$ be a normed linear space.
A Busemann functional associated with a unit vector $u$ in $X$ is denoted by 
$\h^u$ and defined as
$$
\h^u (x) := 
\lim_{t\to+\infty}(\norm{x-tu}-t),
$$
for all $x$ in $X$. This limit exists for each $x$ in $X$ because the function 
$t\mapsto\norm{x-tu} - t\,$ defined on the interval $[0,+\infty)$ is monotone 
non-increasing and bounded below by $-\norm{x}$.  
\end{defn}

\begin{rem}
Every Busemann functional is a metric functional.
\end{rem}

Some of the following statements are probably well-known. We recall them here for 
the reader's convenience.

\begin{prop}\label{prop: Bf1}
Let $(X,\norm{\cdot})$ be a normed linear space.
For every $v$ in $X$ there are two Busemann functionals $\h^{u_1}$ and $\h^{u_2}$ 
such that 
$$
	\h^{u_1}(v) = \norm{v} = -\h^{u_2}(v).
$$
\end{prop}

\begin{proof}
The case $v=0$ is trivial because all metric functionals vanish at $0$. We assume 
now that $v \neq 0$. Consider the unit vector $u_1=(-1/\norm{v})v$. Then, we have
$$
   \h^{u_1}(v) =\lim_{t\to+\infty}(\norm{v-tu_1}-t) 
        =\lim_{t\to+\infty}(\abs{\norm{v}+t} - t) 
        =\norm{v}.
$$
The unit vector $u_2 = -u_1$ gives $\h^{u_2}(v)=-\norm{v}$.
\end{proof}

\begin{prop}\label{prop:NormAsSupOfBusm}
Let $(X,\norm{\cdot})$ be a normed linear space.
For every vector $v$ in $X$ we have
$$
	\norm{v} = \sup\{\,\h(v)\, \mid \h\;\text{is a Busemann functional on}\; X\}
$$
\end{prop}
\begin{proof}
By Proposition~\ref{prop: Bf1}, for every $v$ there is a Busemann functional $\h^u$
on $X$ such that $\h^u(v)=\norm{v}$. Thus, we have
$$
	\sup\{\,\h(v)\, \mid \h\;\text{is a Busemann functional on}\; X\}
	\geq \h^u(v) = \norm{v}.
$$
The other inequality follows from the fact that every Busemann functional is a
metric functional.
\end{proof}

\begin{rem}
The previous statement is valid in geodesic metric spaces which have the property
that every geodesic segment can be extended to a ray. This was observed in
\cite[Section~6]{Karlsson2021}.
\end{rem}

\begin{prop}
Let $(X,\norm{\cdot})$ be a normed linear space.
Every Busemann functional is subadditive and positively homogeneous. In other 
words, if $\h^u$ is a Busemann functional on $X$, then we have 
\begin{enumerate}
\item $\,\h^u(x+y) \leq \h^u(x) + \h^u(y),\;$ for all $x$ and $y$ in $X$, and
     
\item $\,\h^u(sx)=s\h^u(x),\;$ for all $x$ in $X$ and for all $s\geq 0$.
\end{enumerate}
\end{prop}

\begin{proof}
By the definition of a Busemann functional $\h^u$, we have 
$$
	\h^u(x)=\inf_{t\geq 0}(\norm{x-tu}-t),
$$ 
for all $x\in X$. To complete the proof, it suffices to notice that if $x$ and 
$y$ are vectors in $X$ and $s > 0$, then for every $t\geq 0$ we have
$$ \h^u(x+y) \leq \norm{x+y - 2t} - 2t
\leq (\norm{x-tu}-t) + (\norm{y-tu}-t)$$
and
$$ \norm{sx - tu} - t = s (\norm{x-(t/s)u}-(t/s)).$$
\end{proof}

It is well-known that the Hahn-Banach theorem implies that continuous linear 
functionals separate points in normed linear spaces. A natural question is 
whether metric functionals separate points. Internals trivially do this. So, we
are really looking for non-internal metric functionals that separate points.

Karlsson presents a metric Hahn-Banach statement in 
\cite[Proposition~1]{Karlsson2021}. But this result does not immediately give
a separation property. In normed linear spaces, one only needs the definition of 
a Busemann functional (Definition~\ref{def:Busm}) to obtain the following 
separation statement.

\begin{thm}
Let $(X,\norm{\cdot})$ be a normed linear space.
If $x$ and $y$ are two distinct vectors in $X$, then there exists a Busemann 
functional $\,\h^u$ such that $\,\h^u(x)\neq \h^u(y)$.
\end{thm}

\begin{proof}
Let $u$ be the unit vector given by $\norm{x-y}u = x-y$. We note that 
$\norm{x-y-tu} = \abs{\norm{x-y}-t}$ for all $t$. Then, the Busemann functional 
$\h^u$ associated with $u$ takes the value $-\norm{x-y}$ at the vector $x-y$. 
We are going to show that
$$
	\h^u(x) = -\norm{x-y} + \h^u(y).
$$
Indeed, subadditivity of $\h^u$ implies $\h^u(x) \leq -\norm{x-y} + \h^u(y)$. 
The other inequality always holds because $\h^u$ is a metric functional.
\end{proof}

If we use the Hahn-Banach theorem, we can obtain explicit formulas for Busemann 
functionals. For a given unit vector $u$, the set of subdifferentials 
\begin{displaymath}
	\partial \norm{u} = \{ f \in X^* 
	\mid \dpair{u}{f} 
	= 1,\ \norm{f} = 1 \}
\end{displaymath}
is a nonempty convex and closed subset of the dual space $X^*$. With this 
information available we can show the following.

\begin{thm}[{\cite[Proposition~4.12]{Gutierrez-Nevanlinna2024}}]
Let $(X,\norm{\cdot})$ be a normed linear space.
Let $u$ be a unit vector in $X$. Then, the Busemann functional $\h^u$ associated 
with $u$ has the form
\begin{displaymath}
	\h^u(x) = \max_{f \in \partial \norm{u}} \dpair{-x}{f},\quad
	\text{for all}\;x\;\text{in}\;X.   
\end{displaymath}
\end{thm}

\section{Details of $d-$weak convergence}\label{sec:WC}
To come up with our concept of weak convergence proposed in 
Definition~\ref{def:WeakConv}, we first looked at the behavior of some sequences
in metric spaces where explicit formulas for all metric functionals were 
available. We started on our analysis in the $\ell_p$ spaces with 
$1\leq p< \infty$, because in these spaces explicit formulas for all metric 
functionals were at our disposal, see \cite{Gutierrez2019-2}. So, we thought that 
if we were to use metric functionals to broadly test weak convergence, it should 
agree with the standard weak convergence. This was indeed the case as we will 
show in the next section.

Let us focus on our Definition~\ref{def:WeakConv}. Notice that we need all the 
metric functionals on $(X,d)$ to test $d-$weak convergence. It is also important 
to notice that only a limit inferior and an inequality appear in our notion of 
$d-$weak convergence. A reason for this choice is that a whole limit of the form 
$\h(x_n)\rightarrow \h(z)$ with internals $\h$ would obviously give strong limits. 
This was briefly discussed in \cite[Section~6]{Karlsson2021}.

After a quick look at the construction of metric functionals given in 
Section~\ref{ssec:Construction}, 
one may naturally wonder whether choosing a basepoint other than $o$ could 
completely modify $d-$weak convergence. We show next that the basepoint can be 
chosen freely.

\begin{prop} 
Let $(X,d)$ be a metric space. The $d-$weak convergence of sequences in $X$ does 
not depend on the choice of the basepoint upon which all the metric functionals 
on $X$ are built.
\end{prop}

\begin{proof}
We assume that the point $o$ in $X$ is the original basepoint as done in 
Section~\ref{ssec:Construction}. Choose now another basepoint $b$ in $X$. 
Let {\boldmath$\eta$} be a metric functional on $X$ that is built on the 
basepoint $b$. Then, there exists a metric functional $\h$ on $X$ that is 
built on the basepoint $o$ and such that 
$$
	\h(x) = \mbox{\boldmath$\eta$}(x) + \h(b),
$$
for all $x$ in $X$. This follows straightforwardly from the identity
$$
	d(\cdot,w) - d(o,w) = d(\cdot,w) - d(b,w) + [ d(b,w) - d(o,w)].
$$
Now, if a sequence $(x_n)_{n\geq 1}$ in $X$ converges $d-$weakly to a point $z$ 
in $X$, then we have
$$
	\liminf_{n\to\infty}\mbox{\boldmath$\eta$}(x_n) 
		= \liminf_{n\to\infty}\h(x_n) -\h(b)
		\geq \h(z) - \h(b) = \mbox{\boldmath$\eta$}(z).
$$ 
\end{proof} 

\begin{rem}
Suppose that we have $x_n \xrightharpoonup{\;d\;} z$. If one chooses $b:=z$ as the new 
basepoint, then $d-$weak convergence of the same sequence $(x_n)_{n\geq 1}$ reads 
$$
	\liminf_{n\to\infty}\mbox{\boldmath$\eta$}(x_n) \geq 0,
$$
for all metric functionals {\boldmath$\eta$} on $X$, which are now built on the 
basepoint $z$. 
\end{rem}

The result that we present next follows readily from the definition of $d-$weak 
convergence.

\begin{prop}
Let $(X,d)$ be a metric space. If a sequence $(x_n)_{n\geq 1}$ converges 
$d-$weakly to a point $z$ in $X$, then we have
\begin{equation}\label{eq:LimInfBoundedBelow}
	d(z,w) \leq \liminf_{n\to\infty}\,d(x_n,w),
\end{equation}
for all points $w$ in $X$.
\end{prop}

\begin{proof}
Test $x_n \xrightharpoonup{\;d\;} z$ against all the internal metric 
functionals $\h_w$.
\end{proof}

\begin{rem}
In normed linear spaces, the inequality (\ref{eq:LimInfBoundedBelow}) also holds
for sequences that converge weakly (in the standard weak 
topology). This is shown, however, by using the Hahn-Banach theorem.
\end{rem}

\begin{prop}   
Let $p$ and $q$ be two distinct points in a metric space $(X,d)$. Consider the 
sequence of points $(x_n)_{n\geq 1}$ in $X$ with $x_{2n}=p$ and $x_{2n-1} = q$ 
for all $n\geq 1$. Then, the sequence $(x_n)_{n\geq 1}$ does not converge 
$d-$weakly.
\end{prop}
 
\begin{proof}  
Suppose that the sequence $(x_n)_{n\geq 1}$ converges $d-$weakly to a point $z$ 
in $X$. We are going to show a contradiction with the help of both internals 
$\h_p$ and $\h_q$. Indeed, since we have $\h_p(x) > \h_p(p)$ for all $x\neq p$, 
our assumption would imply $z = p$. Then, we would have $\h_q(x_{2n-1}) < \h_q(z)$ 
for all integers $n\geq 1$. Thus, the sequence $(x_n)_{n\geq 1}$ cannot converge 
to $z$.
\end{proof}
  
\begin{prop}
Let $X$ be a set that contains infinitely many points and is equipped with the 
discrete metric $d$. Let $(x_n)_{n\geq 1}$ be a sequence in $X$. Then, we have 
the following possibilities:
\begin{enumerate}
    \item The sequence $(x_n)_{n\geq 1}$ is eventually constant. In other words, 
    the sequence converges strongly.
    \item If there are two subsequences $(x_{m_i})_{i\geq 1}$ and 
    $(x_{n_i})_{i\geq 1}$ where $x_{m_i}=p$ and $x_{n_i}=q$ 
    with $p \not=q$, then $(x_n)_{n\geq 1}$ does not converge $d-$weakly.
    \item If there exists only one point $z$ such that $x_n=z$ 
    happens infinitely often, then $(x_n)_{n \geq 1}$ converges $d-$weakly 
    to $z$ but it need not converge strongly.
    \item If there exists no $z$ for which $x_n=z$ infinitely many 
    times, then $(x_n)_{n \geq 1}$ converges $d-$weakly to all points in $X$.
\end{enumerate} 

\begin{proof}
We only need to prove the last two possibilities. This is done after noting that 
every metric functional on this space $(X,d)$ is either internal or the constant 
functional identically zero.
\end{proof}
\end{prop}

\begin{prop}
Let $(X,d)$ be a metric space. Then, every closed ball 
$$ B(q, r) = \{x\in X  \mid  d(x,q) \leq r \} $$ 
is $d-$weakly sequentially closed.
\end{prop}

\begin{proof}  
Suppose that $(x_n)_{n\geq 1}$ is a sequence in $B(q,r)$ and $z$ is a point in 
$X$ such that $x_n \xrightharpoonup{\;d\;} z$. Here $d-$weak convergence must 
be tested against all the metric functionals on the whole space $X$. We are 
going to show that $z$ is in $B(q,r)$. Indeed, if the point $z$ were not in 
$B(q,r)$, we would have $d(z,q) > r$. Thus, testing $d-$weak convergence against 
the internal $\h_q$ would imply the contradiction
$$ 
	r-d(o,q)\geq \liminf_{n\to\infty}\,\h_q(x_n) 
	\geq \h_q(z) >  r - d(o,q). 
$$
\end{proof}

\begin{rem}
Let $(X,d)$ be a metric space and $A$ be a subset of $X$. If we define the set 
$hull(A)$ as the intersection of all closed balls containing $A$, then $hull(A)$ 
is $d-$weakly sequentially closed.
\end{rem} 

\begin{ex}  
Let $X$ be a nonempty set equipped with the discrete metric. For a nonempty subset 
$A$ of $X$ we have 
$$
	hull(A) = \begin{cases} 
		A & |A| = 1 \\
		X & |A| > 1.
	\end{cases}
$$
Here $|A|$ means the number of elements of $A$.
\end{ex}

Let us now discuss relations between $d-$weak convergence and strong convergence.

\begin{prop}
Let $(X,d)$ be a metric space. If a sequence $(x_n)_{n\geq 1}$ in $X$ converges 
strongly, then it converges $d-$weakly.
\end{prop}

\begin{proof}
Suppose that the sequence $(x_n)_{n\geq 1}$ converges strongly to a point $z$ in 
$X$. Let $\h$ be a metric functional on $X$. Since all metric functionals are 
$1-$Lipschitz, we have $\h(x_n) \geq -d(x_n,z) + \h(z)$ for all $n\geq 1$. From 
this we get 
$$
	\liminf_{n\to\infty}\h(x_n)\geq \h(z).
$$
\end{proof}

When does $d-$weak convergence imply strong convergence? 

In general, a sequence may have more than one $d-$weak limit. The following
result states that there are special cases where we have a unique $d-$weak 
limit which is also a strong one.

\begin{prop}
Suppose that $(X,d)$ is such that all its closed balls are compact. If a bounded
sequence in $X$ converges $d-$weakly, then it does so strongly.
\end{prop}
\begin{proof}
Let $(x_n)_{n\geq 1}$ be a bounded sequence in $X$ such that 
$x_n \xrightharpoonup{\;d\;} z$. Suppose that there exists $\epsilon>0$ such that 
$d(x_n,z)>\epsilon$ for infinitely many $n$. Since all closed balls are compact, 
there is a subsequence $(x_{n_i})_{i\geq 1}$ and a point $p$ in $X$ such that 
$d(x_{n_i},p) \rightarrow 0$. By testing $d-$weak convergence on the internal 
$\h_p$, we get
$$
0\geq \liminf_{n\to\infty}\,d(x_n,p) \geq d(z,p).
$$
Thus, we have $z=p$, which contradicts the existence of such $\epsilon$.
\end{proof}

A quick inspection of the explicit formulas for all metric functionals on 
$\ell_1$ shown in \cite[Theorem~3.6]{Gutierrez2019-2} reveals the following.

\begin{prop}\label{prop:dWeakStrongl1}
If a sequence in $\ell_1$ converges $d-$weakly, then it converges strongly. 
\end{prop}

\begin{proof}
From the formulas shown in \cite[Theorem~3.6]{Gutierrez2019-2}, we notice that 
some of the metric functionals on $\ell_1$ are linear: those are of the form
\begin{equation}\label{eq:LinearMFl1}
	\h(x) = \sum_{k\in I}\varepsilon(k) x(k),
\end{equation}
where $I$ is a nonempty subset of $\N$ and $\varepsilon(k)\in\{-1,1\}$ for all 
$k$ in $I$. Now, let $(y_n)_{n\geq 1}$ be a sequence in $\ell_1$ that converges
$d-$weakly to a point $y$ in $\ell_1$. Thus, for every $k\in \N$ we 
have 
$$
	\lim_{n\to\infty}\,y_n(k) = y(k).
$$
Define $x_n := y_n - y$ and assume that there exists a positive real number 
$\epsilon$ such that $\limsup_{n\to\infty}\norm{x_n} > \epsilon$. What we do next 
is to apply the so-called gliding hump technique. We follow the exact details 
that appear in the book of B.~Beauzamy \cite[pp.~117-118]{Beauzamy1982}.
The gliding hump technique gives a sequence of positive integers 
$(n_p)_{p \geq 1}$ and a sequence $(c(m))_{m\geq 1}$ in $\{-1,0,1\}$ such that
\begin{equation}\label{eq:ConseqGlidingHump}
	\Big|\sum_{m\geq 1}\,c(m)x_{n_p}(m)\Big| 
	\geq \epsilon / 4,\quad\text{if}\quad p\geq 3.
\end{equation}

On the other hand, if we define $I:=\{k\in\N \mid c(k)\neq 0\}$ and 
$\varepsilon(k):=c(k)$ for all $k$ in $I$, we obtain a metric functional $\h$ of 
the form (\ref{eq:LinearMFl1}). Clearly, $-\h$ is also a metric functional of 
the form (\ref{eq:LinearMFl1}). By testing $d-$weak convergence of 
$(y_n)_{n\geq 1}$ against these two metric functionals, we get
$$
	\lim_{n\to\infty}\,\sum_{m\in I}\varepsilon(m)y_{n_p}(m) 
	= \sum_{m\in I}\varepsilon(m)y(m).
$$
This contradicts the inequality (\ref{eq:ConseqGlidingHump}). 
\end{proof}

In uniformly convex normed linear spaces the following statement holds.

\begin{prop}
Let $(E,\norm{\cdot})$ be a uniformly convex normed linear space. Assume that 
$X=B_E$, the closed unit ball of $E$. If a sequence $(x_n)_{n\geq 1}$ in $X$ is 
such that $ \norm{x_n}  \rightarrow \norm{\hat x} $ and that for all $w$ in
$X$ we have 
$$
	\liminf_{n\to\infty}\,\h_w(x_n)  \geq \h_w(\hat x),
$$
then $ \norm{ x_n - \hat x }  \rightarrow 0$.
\end{prop}

\begin{proof} 
We may assume that $\hat x \not=0$ as otherwise the claim is trivial. First,
suppose that we have $\norm{x_n} = \norm{ \hat x }$ for all $n\geq 1$, but there 
is a positive real number $\epsilon$ such that
$$
	\limsup_{n\to\infty} \norm{x_n - \hat x }  > \epsilon\norm{\hat x}.
$$ 
By uniform convexity, there would exist a positive real number $\delta$ 
such that 
$$
\norm{ (x_n + \hat x)/2 }  \leq (1-\delta) \norm{\hat x }.
$$
Testing $d-$weak convergence of the sequence $(x_n)_{n\geq 1}$ with the 
internal $\h_{-\hat x}$ gives
$$
\liminf_{n\to\infty} \norm{ ( x_n + \hat x)/2 } \geq \norm{\hat x},
$$ 
which shows that such $\epsilon $ cannot exist. 

Now, if we only have  $ \norm{ x_n } \rightarrow \norm{\hat x }$, we can 
move to $y_n :=  \frac { \norm{ \hat x }}{\norm{x_n} } x_n$ so that 
$\norm{y_n -  x_n} \rightarrow 0$. By testing $d-$weak convergence of the 
sequence $(x_n)_{n\geq 1}$ with the internal $\h_{-\hat x}$, we get 
$$
	\liminf_{n\to\infty}\norm{ y_n + \hat x} 
	\geq \liminf_{n\to\infty}\Big|\norm{x_n + \hat x} - \norm{x_n - y_n} \Big| 
	\geq 2\norm{\hat x}.
$$ 
Then, the limit $\norm{y_n -  \hat x} \rightarrow 0$ holds. Hence, 
we have $\norm{x_n -  \hat x} \rightarrow 0$.
\end{proof} 

\begin{rem}
In general, there are more elements in $(B_E)^{\diamondsuit}$ than just the 
internals. In the previous proposition, testing $d-$weak convergence only 
against internals was enough. Also, the same proof holds in the whole space $E$.
\end{rem}

Next, we want to discuss some behavior of $d-$weak convergence in closed balls
of $\ell_2$ and $\ell_1$. For this purpose, we define the following.

\begin{defn}  
Let $(X,d)$ be a metric space. For a given sequence $(x_n)_{n\geq 1}$ in $X$ we 
define the set
$$
	\Lambda_d(x_n) := \{ z\in X \mid \ x_n \xrightharpoonup{\;d\;} z\}.
$$
\end{defn}

\begin{ex}  
Let $X$ be the closed unit ball of $\ell_2$. Following the construction of metric 
functionals on Hilbert spaces \cite{Gutierrez2019-2}, we note that all metric 
functionals on $X$ are of the form
$$
	\h(x)  = ( \norm{x}^2 - 2\dpair{x}{z} +c^2 )^{1/2} -c,
$$
where $\norm{z} \leq c \leq 1$. For all such metric functionals $\h$ we have  
$$
	\sqrt{5}/2 -1 \leq \lim_{n\to\infty} \h(e_n/2) \leq 1/2.
$$
Thus, we have $\Lambda_d(2^{-1}e_n) = (\sqrt{5}/2 -1)X$. In general, if we 
consider the sequence $(x_n)_{n\geq 1}$ in $X$ with $x_n = \theta e_n$ and 
$0<\theta\leq 1$, then we observe that a point $u$ is in $\Lambda_d(\theta e_n)$ 
if and only if $\norm{u}^2 + 2\norm{u} \leq \theta ^2$. Clearly, we cannot have 
$\norm{x_n} \rightarrow \norm{u}$.
\end{ex} 

\begin{ex}
Suppose that the set $X=\{ 0, e_1, e_2, \cdots \}$ is equipped 
with the metric $d$ induced by the norm of $\ell_1$. All the metric functionals
on $X$ are internal: $\h_0(x)$ and $\h_{e_j}$ for all $j\geq 1$. Thus, we have 
$\Lambda_d(e_n) = X$.
\end{ex}

\begin{ex}  
Let $X$ be the closed unit ball of $\ell_1$. Then, all the metric functionals 
on $X$ are internal \cite{Gutierrez2019-2}, and hence we have 
$\Lambda_d(2^{-1}e_n) = 2^{-1}X$. To see this it suffices to note that for all 
$w$ in $X$ we have
$$
	\h_w(2^{-1}e_n) = \norm{2^{-1}e_n -w}_{1} - \norm{w}_{1} \rightarrow  1/2.
$$
\end{ex}

Before we close this section we want to point out the difference between our notion 
of $d-$weak convergence in metric spaces and the one proposed by T.~C. Lim \cite{Lim1977}; 
the so called $\Delta-$convergence.

A sequence $(x_n)_{n\geq 1}$ in a metric space $(X,d)$ is said to $\Delta-$converge
to a point $z$ in $X$ if  for every $y$ in $X$ one has $d(z, x_n) \leq d(y, x_n) + o(1)$
as $n\to\infty$. It is known that when $X=\ell_2$ the standard weak convergence 
agrees with $\Delta-$convergence for bounded sequences. However, in $X=\ell_1$ these 
concepts do not agree. To see this, consider the sequence $(x_n)_{n\geq 1}$ in
$\ell_1$ where $x_n = e_n$ for all $n\geq 1$. While the sequence does not converge 
$d-$weakly, it $\Delta-$converges to $0$. Thus, the concept of $\Delta-$convergence 
fails to agree with weak convergence of bounded sequences in some Banach spaces. 

\section{Proofs of the main theorems}

\begin{proof}[\bf Proof of Theorem \ref{thm:unique}]
Suppose that there are two distinct vectors $u$ and $v$ in $X$ for which we have
\begin{equation}\label{eq:twolimits}
	x_n \xrightharpoonup{\;d\;} u \quad\text{and}\quad x_n \xrightharpoonup{\;d\;} v. 
\end{equation}
Let $B_{X^*}$ denote the closed unit ball of the dual space $X^*$. Let 
$\mathcal{E}(B_{X^*})$ denote the set of all extreme points of $B_{X^*}$. The
inclusion $\mathcal{E}(B_{X^*}) \subset X^\diamondsuit$ holds as a result of
the statement in \cite[Corollary 3.5]{Walsh2018}. Thus, testing $d-$weak 
convergence in (\ref{eq:twolimits}) against every element $g$ in 
$\mathcal{E}(B_{X^*})$ gives
$$
\lim_{n\to\infty}\,\dpair{x_n}{g}=\dpair{u}{g}=\dpair{v}{g}.
$$
Clearly, we have $\dpair{u}{f}=\dpair{v}{f}$ for all $f$ in the convex hull of
$\mathcal{E}(B_{X^*})$.
Finally, the Krein-Milman theorem implies that the equality 
$\dpair{u}{f}=\dpair{v}{f}$ holds for all $f\in X^*$. This is a contradiction 
because continuous linear functionals separate points of $X$.
\end{proof}

\begin{proof}[\bf Proof of Theorem \ref{thm:dWeakANDstWeak}]
Assume that $(x_n)_{n\geq 1}$ converges weakly to $x$. Suppose now that there is a metric 
functional $\h^* \in X^\diamondsuit$ such that
\begin{equation}\label{eq:StrictIneq}
	\liminf_{n\to\infty} \h^*(x_n) < \h^*(x).    
\end{equation}
Clearly, $\h^*$ is not the trivial zero functional. Denote by $A$ the liminf in the inequality 
(\ref{eq:StrictIneq}). Since $\h^*$ is a metric funtional, we have $\h^*(x_n)\geq -\norm{x_n}$ 
for all $n\geq 1$, hence $A$ is finite. We pick a subsequence $(x_{n_i})_{i\geq 1}$ such that 
$\h^*(x_{n_i})\to A$. Since $A < \h^*(x)$, there exists a positive real number $\epsilon$ such 
that $A+\epsilon < \h^*(x)$. Then, there exists a positive integer $i_0$ such that for every
$i\geq i_0$ we have
\begin{equation}\label{eq:PointsIOX}
	\h^*(x_{n_i}) \leq A +\epsilon <\h^*(x).
\end{equation}
Now, consider the set 
$$ 
	Y = \{ y\in X \mid \h^*(y) \leq A+\epsilon \}.
$$
Since $\h^*$ is a metric functional, $Y$ is a closed convex subset of $X$. From 
(\ref{eq:PointsIOX}) it follows that $Y$ is nonempty and $x\in X\setminus Y$. A classical 
separation theorem implies that there exists a continuous linear functional $f\in X^*$ and a real 
number $B$ such that $\dpair{x_{n_i}}{f} \geq B > \dpair{x}{f}$ for all $i \geq i_0$. This 
contradicts the assumption of weak convergence. So, such a metric functional $\h^*$ that 
satisfies (\ref{eq:StrictIneq}) cannot exists. 
	
Now, we are going to prove the other direction of our claim. We assume that for every metric 
functional $\h\in X^\diamondsuit$ we have
\begin{equation}\label{eq:2SeqLimInf}
	\liminf_{n\to\infty}\h(x_n)\geq \h(x).
\end{equation}
By a result of Walsh (\cite[Corollary 3.5]{Walsh2018}), extreme points of the closed 
unit ball $B_{X^*}$ of $X^*$ are elements in $X^\diamondsuit$. Thus, the inequality (\ref{eq:2SeqLimInf})
holds for all extreme points $\pm a^*$ of $B_{X^*}$. Then, we have
$$ 
	\lim_{n\to\infty}\dpair{-x_n}{a^*} =  \dpair{-x}{a^*}.
$$
Since $(x_n)$ is bounded, the above equality is equivalent to weak convergence due to Rainwater's
theorem \cite{Rainwater1963}.
\end{proof}

\begin{proof}[\bf Proof of Theorem \ref{thm:UnbNoWeak}]
In $\ell_1$, the claim follows from Proposition~\ref{prop:dWeakStrongl1}.

Consider now the space $C[0,1]$ equipped with the sup-norm. For each $t$ in 
$[0,1]$, both functionals $f\mapsto f(t)$ and $f\mapsto -f(t)$ are metric 
functionals on $C[0,1]$. This follows from the explicit formulas given in 
\cite[Theorem~5.1]{Walsh2018}. If a sequence $(f_n)_{n\geq 1}$ in $C[0,1]$ is 
unbounded, then there is a real number $\tau$ in $[0,1]$ and a sequence of 
positive integers $(n_i)_{i\geq 1}$ such that
$$
	\abs{f_{n_i}(\tau)} \rightarrow \infty.
$$
If we have $f_{n_i}(\tau) \rightarrow -\infty $, use the metric functional 
$\h(f)=f(\tau)$ to get $\h(f_{n_i})\rightarrow -\infty$. Thus, the sequence
$(f_n)_{n\geq 1}$ cannot converge $d-$weakly. If we have 
$f_{n_i}(\tau) \rightarrow +\infty $, use the metric functional $\h(f)=-f(\tau)$
instead.

Finally, assume that $X$ is a normed linear space whose dual $X^*$ is
strictly convex. We know that extreme points of the closed unit ball $B_{X^*}$
are metric functionals \cite[Corollary 3.5]{Walsh2018}. It is also known that,
in the present case, all continuous linear functionals of norm 1 are extreme points
of $B_{X^*}$. Thus, every sequence in $X$ that converges $d-$weakly, also converges  
in the standard weak topology, and hence must be bounded. 
\end{proof}

To prove Theorem~\ref{thm:dWeakOfLinComb} we need the following result.

\begin{lem}\label{lem:MFcomb}
Let $(X,\norm{\cdot})$ be a normed linear space. Suppose that $s$ and $t$ are
real numbers and $v$ is a vector in $X$. If $\h$ is a metric functional on $X$,
then there exists a metric functional {\boldmath$\eta$} on $X$ such that 
$$
	\h(sx + tv) = \abs{s}\mbox{\boldmath$\eta$}(x) + \h(tv),
$$
for all $x$ in $X$.
\end{lem} 

\begin{proof}
The claim trivially holds for the case $s=0$. Now, suppose that $s$ is a nonzero 
real number. If $\h$ is a metric functional on $X$, then there is a net 
$(w_\alpha)$ in $X$ such that
$$
	\h(\cdot)= \lim_{\alpha}\, (\norm{\,\cdot\,- w_\alpha} - \norm{w_\alpha}).
$$
Consider the net $(z_\alpha)$ in $X$ with $z_\alpha = s^{-1}(w_\alpha - tv)$. 
Then, for every $x$ in $X$ we have
$$
	\norm{sx + tv - w_\alpha}- \norm{w_\alpha} = \abs{s}(\norm{x - z_\alpha} 
		-\norm{z_\alpha}) + (\norm{tv - w_\alpha} 
		- \norm{w_\alpha}).
$$
Thus, our claim holds with
$$
	\mbox{\boldmath$\eta$}(x) = \lim_{\alpha}(\norm{x - z_\alpha} 
	-\norm{z_\alpha}).
$$
\end{proof}

\begin{proof}[\bf Proof of Theorem \ref{thm:dWeakOfLinComb}]
Let $\h$ be an arbitrary metric functional on $X$. By Lemma~\ref{lem:MFcomb}, there 
exists a metric functional {\boldmath$\eta$} on $X$ such that
$$
\h(sx + tv) = \abs{s}\mbox{\boldmath$\eta$}(x) + \h(tv),
$$
for all $x$ in $X$. Thus, the hypothesis $x_n \xrightharpoonup{\;d\;} u$ implies
\begin{align*}
	\liminf_{n\to\infty}\,\h(sx_n + tv) 
		&= \abs{s}\liminf_{n\to\infty}\,\mbox{\boldmath$\eta$}(x_n) + \h(tv) \\
		&\geq \abs{s}\mbox{\boldmath$\eta$}(u) + \h(tv) \\
		&= \h(su + tv). 
\end{align*}
Now, since $\h$ is a metric functional and the hypothesis $d(y_n,v) \rightarrow 0$ 
holds, we have
$$
	\liminf_{n\to\infty}\,\h(sx_n + ty_n) = \liminf_{n\to\infty}\,\h(sx_n + tv).
$$ 
Therefore, we have $sx_n + ty_n \xrightharpoonup{\;d\;} su + tv$.
\end{proof}

\begin{proof}[\bf Proof of Theorem \ref{thm:dWeakInWconvex}]
	
Denote by $\Lambda_d(x_n)$ the set of all points $z$ such that
$x_n \xrightharpoonup{\;d\;} z$.
We prove first that the set $\Lambda_d(x_n)$ is convex. Let $z_1$ and $z_2$ be
two points in $\Lambda_d(x_n)$ and $s$ be a real number in the interval 
$[0,1]$. Let $\h$ be a metric functional on the $W-$convex metric space $(X,d)$.
By Proposition~\ref{prop:ConvexMF}, the metric functional $\h$ is $W-$convex. Thus, 
we have
\begin{align*}
	\h(W(z_1,z_2,s)) &\leq (1-s)\h(z_1) + s\h(z_2) \\
	&\leq (1-s)\liminf_{n\to\infty}\,\h(x_n)
	+ s\liminf_{n\to\infty}\,\h(x_n) \\
	&\leq \liminf_{n\to\infty}\,\h(x_n).
\end{align*}
This shows that $x_n \xrightharpoonup{\;d\;} W(z_1,z_2,s)$ for all 
$s\in [0,1]$. Therefore, the set $\Lambda_d(x_n)$ is convex.
	
To show that the set $\Lambda_d(x_n)$ is closed we need to recall that
every metric functional is $1-$Lipschitz. If $(z_m)_{m\geq 1}$ is a 
sequence of points in $\Lambda_d(x_n)$, then 
for every metric functional $\h$ and every $m\geq 1$ we have
$$
	\liminf_{n\to\infty}\,\h(x_n) \geq \h(z_m).
$$
So, if we have $d(z_m,z)\to 0$, then $z$ is in $\Lambda_d(x_n)$.
\end{proof}

%\begin{thebibliography}{10}
	
%\end{thebibliography}
\bibliographystyle{alpha}
\bibliography{dWeakConvergence.bib}
\end{document}